\documentclass[11pt]{amsart}
\usepackage{yfonts}
\usepackage[all]{xy}
\usepackage{amsfonts,amscd,amssymb,amsthm}
\usepackage{amsmath,amsxtra,eucal,graphicx,graphics}
\usepackage{tikz,float}
\usepackage{enumerate}
 \usepackage[T1]{fontenc}
\usepackage[utf8]{inputenc}
\newcommand{\overbar}[1]{\mkern 1.5mu\overline{\mkern-1.5mu#1\mkern-1.5mu}\mkern 1.5mu}
\newcommand{\cupdot}{\mathbin{\mathaccent\cdot\cup}}

\theoremstyle{definition}
\newtheorem{theorem}{Theorem}[section]

\newtheorem{lemma}[theorem]{Lemma}

\newtheorem{conjecture}[theorem]{Conjecture}
\newtheorem{question}[theorem]{Question}

\title{On the average order of a dominating set of a forest}
 
\author{Aysel Erey}
\address{Department of Mathematics\\
Gebze Technical University\\
Kocaeli, Turkey}
\email{aysel.erey@gtu.edu.tr} 

 \date{\today}

\begin{document}

\begin{abstract}
We show that the average order of a dominating set of a forest graph $G$ on $n$ vertices with no isolated vertices is at most $2n/3$. Moreover, the equality is achieved if and only if every non-leaf vertex of $G$ is a support vertex with one or two leaf neighbors. Our result answers an open question of Beaton and Brown \cite{brown_beaton}.
\end{abstract}


\maketitle

\section{Introduction}

In this article all graphs are finite, simple, loopless and undirected. Given a graph $G$, let $V(G)$ and $E(G)$ be the vertex set and the edge set of $G$, respectively. The \textit{order} of $G$ is $|V(G)|$ and the \textit{size} of $G$ is $|E(G)|$. A vertex $u$ is a {\it neighbor} of vertex $v$ in $G$ if $u$ and $v$ are adjacent in $G$. The {\it open neighborhood of v}, $N_G(v)$, consists of all neighbors of $v$ in $G$, and the {\it closed neighborhood of v}, $N_G[v]$, is equal to $N_G(v)\cup \{v\}$.  A vertex $v$ of $G$ is called an {\it isolated vertex} of $G$ if $v$ has no neighbors in $G$.  For a subset $S$  of vertices of $G$,  let $G\setminus S$ denote the subgraph induced by the vertices of $V(G)\setminus S$ (if $S=\{v\}$ is a singleton, we simply write $G\setminus v$). If $H$ is a subgraph of $G$, we write $G\setminus H$ for the subgraph induced by $V(G)\setminus V(H)$ in $G$. Also, let $G/u$ be the graph obtained from $G$ by deleting the vertex $u$ and adding edges between all pairs of nonadjacent neighbors of $u$. The complete graph, empty graph and star graph on $n$ vertices are denoted by $K_n$, $\overbar{K}_n$  and $K_{1,n-1}$, respectively. A vertex of degree one is called a {\it leaf} and an edge containing a leaf vertex is called a {\it pendant edge}. Let $L_G(u)$ denote the set of all leaf neighbors of a vertex $u$ in $G$ and $L_G[u]=L_G(u)\cup\{u\}$. We say that $u$ is a {\it support vertex} of $G$ if $u$ is adjacent to a leaf vertex $v$ of $G$, and $u$ is called the {\it support of} $v$ in $G$. An acyclic graph is callled a \textit{forest} and an acyclic connected graph is called a \textit{tree}.

Average values of various graph invariants have been studied in the literature.  In $1971$, Doyle and Graver \cite{doyle} initiated the study of mean (average) distance in a graph which received a considerable attention. They gave a formula for computing the mean distance of trees, and determined extremal graphs with maximum or minimum mean distance among connected graphs of a given order \cite{doyle}. Extremal problems for the mean distance were also examined in certain other classes of graphs \cite{rautenbach}, \cite{hendry}. Chung \cite{chung_indep} showed that the independence number is an upper bound the mean distance and  Dankelmann \cite{dankelmann} proved upper bounds for it in terms of the so called $k$-packing number.  Average eccentricty (a distance related notion), was recently studied in \cite{dankelmann}. Average order of  a subtree of a graph was introduced by Jamison~\cite{jamison} in $1983$ and this invariant have been investigated by a number of researchers, see, for example, \cite{cameron, mol, vince, wagner}. Moreover, the average connectivity of a graph was  considered in \cite{avg_conn}. Most recently, Andriantiana et al. \cite{and_1,and_2} studied the average sizes of independent sets and matchings. Lastly, The average distance \cite{chung} and the average size of independent sets \cite{davies} were also studied in the context of random graphs.

Our focus in this paper will be on the average order of a dominating set of a graph. A subset of vertices $S$ is called a {\it dominating set} of $G$ if every vertex in $V(G)\setminus S$ is adjacent to some vertex in $S$. Let $\mathcal{D}(G)$ denote the family of all dominating sets of $G$. Recently Beaton and Brown \cite{brown_beaton} introduced the \textit{average order of a dominating} set of $G$, denoted $\operatorname{avd}(G)$, which is given by
$$\operatorname{avd}(G)=\dfrac{\sum\limits_{S\in \mathcal{D}(G)}|S|}{|\mathcal{D}(G)|}.$$
They showed that the complete graph $K_n$ uniquely minimizes the average order of a dominating set among all graphs on $n$ vertices \cite{brown_beaton}. It is trivial that $\operatorname{avd}(\overbar{K}_n)=n\geq \operatorname{avd}(G)$ for every graph $G$ on $n$ vertices. So, the empty graph $\overbar{K}_n$ has the largest average order of a dominating set among graphs on $n$ vertices. What if we do not allow isolated vertices?  Which graphs have the largest average order of a dominating set among all graphs of order $n$ without isolated vertices? While this question was examined in \cite{brown_beaton}, the question remained unanswered in general.  It was shown that $\operatorname{avd}(G)\leq \frac{3}{4}n$ for every graph $G$ of order $n$ without isolated vertices \cite{brown_beaton}. However the factor $3/4$ in the upper bound is not best possible, and the evidence provided in \cite{brown_beaton} suggests that the constant $3/4$ can be improved to a smaller number. Indeed, the following conjecture was proposed.
\begin{conjecture}\label{brown_conj}\cite{brown_beaton}
If $G$ is a graph of order $n$ with no isolated vertices, then $\operatorname{avd}(G)\leq \frac{2n}{3}.$
\end{conjecture}

Conjecture~\ref{brown_conj} was verified for all graphs up to $9$ vertices, all graphs with minimum degree at least $4$, and  all quasi-regularizable graphs (which include all graphs containing a perfect matching or a hamiltonian cycle) \cite{brown_beaton}. Beaton and Brown \cite{brown_beaton} also studied such extremal problems within the family of trees. They showed that for every tree graph $G$ of order $n$ with $G\ncong K_{1,n-1}$,
$$\operatorname{avd}(G)>\frac{n-1+2^{n-2}(n+1)}{2^{n-1}+1}=\operatorname{avd}(K_{1,n-1})$$
and hence the star graph $K_{1,n-1}$ is the unique extremal graph with minimum average order of a dominating set. On the other hand, the problem of determining extremal graphs maximizing this parameter among trees remained as an open problem. In this article, we settle this problem by proving the following:

\begin{theorem}\label{tree_main}
If $G$ is a tree of order $n\geq 2$, then $\operatorname{avd}(G)\leq \frac{2n}{3}.$  Moreover, the equality holds if and only if every non-leaf vertex of $G$ is a support vertex with one or two leaf neighbors. 
\end{theorem}

We actually prove Conjecture~\ref{brown_conj} for forests, see our Theorem~\ref{main} which immediately implies Theorem~\ref{tree_main}.

\section{Preliminaries}

The \textit{domination number} $\gamma (G)$ of a graph $G$ is the cardinality of a minimum dominating set of $G$. Let $d_k(G)$ be the number of dominating sets of $G$ with cardinality $k$. The \textit{domination polynomial} of $G$, denoted by $D_G(x)$, is given by
$$D_G(x)=\sum\limits_{k=\gamma (G)}^{|V(G)|}d_k(G)\, x^k.$$

Observe that $\operatorname{avd}(G)$ is equal to the logarithmic derivative of $D_G(x)$ evaluated at $x=1$, that is,
$$\operatorname{avd}(G)=\dfrac{D'_G(1)}{D_G(1)}.$$

 Let $H_1,\dots , H_c$ be the connected components of $G$. It is well known that 
$$D_G(x)=\prod\limits_{i=1}^{c}D_{H_i}(x),$$
and moreover, it was observed in \cite{brown_beaton} that
$$\operatorname{avd}(G)=\sum\limits_{i=1}^{c}\operatorname{avd}(H_i).$$

Let  $V(G)=\{v_1,\dots , v_n\}$ and let $G(v_1^{k_1},v_2^{k_2},\dots , v_n^{k_n})$ be the graph obtained from $G$ by adding $k_i$ leaves to each vertex $v_i$ for $i\in\{1,\dots ,n\}$. For $k_1,\dots , k_n>0$,  it was observed in \cite{akbari} and  \cite{oboudi}  that
$$D_{G(v_1^{k_1},v_2^{k_2},\dots , v_n^{k_n})}(x)=\prod\limits_{i=1}^nD_{K_{1,k_i}}(x).$$

Thus, $\operatorname{avd}(G(v_1^{k_1},v_2^{k_2},\dots , v_n^{k_n}))=\sum\limits_{i=1}^n\operatorname{avd}(K_{1,k_i})$. In particular, if $G'$ is obtained from $G$ by adding one leaf to each of $p$ vertices of $G$ and two leaves to each of $q$ vertices of $G$, where $|V(G)|=p+q$, then
$$\operatorname{avd}(G')=p\operatorname{avd}(K_{1,1})+q\operatorname{avd}(K_{1,2})=\frac{4}{3}p+2q=\frac{2}{3}(2p+3q)=\frac{2}{3}|V(G')|.$$
Thus, we see that the extremal graphs mentioned in Theorem~\ref{tree_main} indeed achieve the given upper bound.

We will also make use of the following recursive formula for the domination polynomials of graphs containing vertices with nested closed neighborhoods. 

\begin{lemma}\cite{kotek}\label{recur} Let $u$ and $v$ be two vertices of $G$ such that $N_G[v]\subseteq N_G[u]$. Then,
$$D_G(x)=xD_{G/u}(x)+D_{G\setminus u}(x)+xD_{G\setminus N_G[u]}(x).$$
In particular, if $v$ is a leaf vertex and $u$ is its neighbor, then
$$D_G(x)=x\left[D_{G/u}(x)+D_{G\setminus \{u,v\}}(x)+D_{G\setminus N_G[u]}(x)\right].$$
\end{lemma}

Lastly, observe that if $H$ is a subgraph of a graph $G$, then $D_H(1)\leq D_G(1)$ because the function $f:\mathcal{D}(H)\rightarrow \mathcal{D}(G)$ defined by $f(S)=S\cup (V(G)\setminus V(H))$ is clearly one-to-one.
\section{Main Result}

In the proof of the following result, let $d_k(G,u)$ (respectively $d_k(G,\overbar{u}$) denote the number of dominating sets of $G$ of order $k$ which contain the vertex $u$ (respectively do not contain the vertex $u$). Clearly, $d_k(G)=d_k(G,u)+d_k(G,\overbar{u})$ for every vertex $u$.

\begin{lemma}\label{support_lemma} Let $G$ be a graph of order $n$ and $w$ be a support vertex of $G$ with $L_G(w)=\{v_1,\dots ,v_t\}$ for some integer $t\geq 1$. Let also $H=G\setminus L_G[w]$.  Suppose that every vertex $u$ in $N_G(w)\setminus L_G(w)$ is a support vertex in $G$, and $3D'_H(1)\leq 2(n-t-1)D_H(1)$. Then, 
$$3D'_G(1)\leq 2nD_G(1)$$
with equality if and only if $t\in\{1,2\}$ and $3D'_H(1)=2(n-t-1)D_H(1)$.
\end{lemma}

\begin{proof}
Let $S$ be a dominating set of $G$ of order $k$. If $v_t\notin S$, then $w$ must be in $S$, as $v_t$ is a leaf. Since every vertex in $N_G(w)\setminus L_G(w)$ is a support vertex of $G$, the vertex subset $S\cap V(H)$ must be a dominating set of $H$. Hence, if $t=1$, then $d_k(G,\overbar{v_t})=d_{k-1}(H)$. Also, if $t>1$, then $S\cap V(H)$ is a dominating set of $H$ on $k-1-i$ vertices where $i$ is the number of vertices in $S\cap\{v_1,\dots ,v_{t-1}\}$. Therefore,
$$d_k(G,\overbar{v_t})=\sum_{i=0}^{t-1}{t-1 \choose i}d_{k-1-i}(H)$$
where ${r\choose 0}=1$ for every integer $r\geq 0$. Similarly, one can check that the number of dominating sets of $G$ of order $k$ which contain $v_t$ but not $w$ is $d_{k-t}(H)$, and  the number of the ones which contain both $v_t$ and $w$ is $\sum_{i=0}^{t-1}{t-1 \choose i}d_{k-2-i}(H)$. Hence,
$$d_k(G,v_t)=d_{k-t}(H)+\sum_{i=0}^{t-1}{t-1 \choose i}d_{k-2-i}(H).$$
Since $d_k(G)=d_k(G,v_t)+d_k(G,\overbar{v_t})$, we have
$$d_k(G)=d_{k-t}(H)+\sum_{i=0}^{t-1}{t-1 \choose i}d_{k-2-i}(H)+\sum_{i=0}^{t-1}{t-1 \choose i}d_{k-1-i}(H).$$
We write the latter as a polynomial equation as follows:
$$D_G(x)=x^tD_H(x)+\sum_{i=0}^{t-1}{t-1\choose i} \left[x^{i+2}+x^{i+1}\right]D_H(x)$$
and differentiating $D_G(x)$, we get
$$D'_G(x)=tx^{t-1}D_H(x)+x^tD'_H(x)+\sum_{i=0}^{t-1}{t-1\choose i}\Big(\left[(i+2)x^{i+1}+(i+1)x^i\right]D_H(x)+\left[x^{i+2}+x^{i+1}\right]D'_H(x)\Big).$$
We evaluate both $D(G,x)$ and $D'(G,x)$ at $x=1$ and obtain that
$$D_G(1)=D_H(1)+\sum_{i=0}^{t-1}2{t-1\choose i}D_H(1)$$
and
$$D'_G(1)=tD_H(1)+D'_H(1)+\sum_{i=0}^{t-1}{t-1\choose i}\left[(2i+3)D_H(1)+2D'_H(1)\right].$$
Now we shall consider two cases:

\noindent {\it Case 1:} $1\leq t \leq 2$. In this case we calculate that  $$2nD_G(1)-3D'_G(1)=(2t+1)\left[2(n-1-t)D_H(1)-3D'_H(1)\right]\geq 0.$$ So,  $2nD_G(1)\geq 3D'_G(1)$ holds with equality if and only if $2(n-t-1)D_H(1)=3D'_H(1)$.

\noindent {\it Case 2:} $t\geq 3$.  In this case we shall prove the strict inequality $3D'_G(1)<2nD_G(1)$. By the assumption, we have  $3D'_H(1)\leq 2(n-t-1)D_H(1)$ and so, it suffices to show that
$$tD_H(1)+6\sum_{i=0}^{t-1}{t-1\choose i}D'_H(1)< 2D_H(1)+\sum_{i=0}^{t-1}{t-1\choose i}[4n-6i-9]D_H(1).$$

Note that ${t-1\choose i}={t-1\choose t-1-i}$. First suppose that $t-1$ is odd. Then,

\begin{eqnarray*}
\sum_{i=0}^{t-1}{t-1\choose i}[4n-6i-9]D_H(1) &=& \sum_{i=0}^{(t-2)/2}{t-1\choose i}[8n-6t-12] D_H(1)\\
&= &   \sum_{i=0}^{(t-2)/2}{t-1\choose i}[8(n-t-1)+(2t-4)] D_H(1)\\
&\geq & \sum_{i=0}^{(t-2)/2}12{t-1\choose i}D_H'(1)+\sum_{i=0}^{(t-2)/2}{t-1\choose i}(2t-4)D_H(1)\\
&=&6\sum_{i=0}^{t-1}{t-1\choose i}D_H'(1)+\sum_{i=0}^{(t-2)/2}{t-1\choose i}(2t-4)D_H(1)\\
&>&6\sum_{i=0}^{t-1}{t-1\choose i}D_H'(1)+tD_H(1).
\end{eqnarray*}
Now suppose that $t-1$ is even.  Then,
$$\sum_{i=0}^{t-1}{t-1\choose i}[4n-6i-9]={t-1\choose (t-1)/2}[4n-3t-6]+\sum_{i=0}^{(t-3)/2}{t-1\choose i}[8n-6t-12]$$ and
$$6\sum_{i=0}^{t-1}{t-1\choose i}=6{t-1\choose (t-1)/2}+\sum_{i=0}^{(t-3)/2}12{t-1\choose i}.$$
By the assumption that $3D'_H(1)\leq 2(n-t-1)D_H(1)$, it is easy to see that
$$\displaystyle \sum_{i=0}^{(t-3)/2}{t-1\choose i}[8n-6t-12]D_H(1)>\sum_{i=0}^{(t-3)/2}12{t-1\choose i}D'_H(1)$$ and 
$$2D_H(1)+{t-1\choose (t-1)/2}[4n-3t-6]D_H(1)>6{t-1\choose (t-1)/2}D'_H(1)+tD_H(1).$$
\end{proof}

Given a graph $G$  with a specified vertex $u\in V(G)$, we write $G_{(u,k)}$ to denote the graph obtained by gluing $G$ and  $K_{k+1}$ at the vertex $u$. That is, $G_{(u,k)}=K_{k+1}\cup G$ and $K_{k+1}\cap G=\{u\}$.

\begin{lemma}\label{glue_lem} Let $G$ be a graph and $u\in V(G)$. Then, for every integer $k\geq 1$,
$$D_{G_{(u,k)}}(x)=(x+1)^{k-1}\left[D_{G_{(u,1)}}(x)+D_{G\setminus u}(x)\right]-D_{G\setminus u}(x).$$
\end{lemma}
\begin{proof} We proceed by induction on $k$. The statement is clear  for $k=1$. Suppose that $k\geq 2$ and let $V(G_{(u,k)})\setminus V(G)=\{v_1,\dots , v_k\}$ where the vertices $v_1,\dots ,v_k$ induce a  $k$-clique in $G_{(u,k)}$ and $u$ is adjacent to each $v_i$ for $1\leq i \leq k$. Since $v_k$ and $v_{k-1}$ have the same closed neighborhoods in $G_{(u,k)}$, we obtain
\begin{eqnarray*}
D_{G_{(u,k)}}(x) &=& x D_{G_{(u,k)}/ v_k}(x)+D_{G_{(u,k)}\setminus v_k}(x)+ xD_{G_{(u,k)}\setminus N_{G_{(u,k)}}[v_k]}(x)\\
&=& xD_{G_{(u,k-1)}}(x)+D_{G_{(u,k-1)}}(x)+xD_{G\setminus u}(x)\\
&=& (x+1)D_{G_{(u,k-1)}}(x)+xD_{G\setminus u}(x)\\
&=& (x+1)\left[(x+1)^{k-2}\left(D_{G_{(u,1)}}(x)+D_{G\setminus u}(x)\right)-D_{G\setminus u}(x)\right]+xD_{G\setminus u}(x)\\
&=& (x+1)^{k-1}\left[D_{G_{(u,1)}}(x)+D_{G\setminus u}(x)\right]-D_{G\setminus u}(x).
\end{eqnarray*}
\end{proof}

\begin{lemma}\label{3_item_lemma} Let $T$ be a tree with $|V(T)|\geq 3$ and $u$ be a vertex of $T$. Suppose that $u$ is not a support vertex of $T$ and $u$ has at most one neighbor in $T$ which  is not a support vertex in $T$. Let $G_1$ be the graph obtained from $T$ by attaching a new leaf vertex $v$ at $u$. Then,
\begin{itemize}
\item[(i)] $D_{G_1}(1)\leq D_T(1)+3D_{T\setminus u}(1)$,
\item [(ii)] $D_{G_1}(1)\leq 5D_{T\setminus u}(1)$ and
\item[(iii)] $D_T(1)\leq 3D_{T\setminus u}(1)$.
\end{itemize}
\end{lemma}
\begin{proof} Given a graph $G$ containing $T\setminus u$ as a subgraph, let us first define $\mathcal{A}_G$ and $\mathcal{B}_G$  as follows:  $\mathcal{A}_G=\{S:\, S\in \mathcal{D}(G) \ \text{and} \ S\cap V(T\setminus u) \in \mathcal{D}(T\setminus u)\}$ and $\mathcal{B}_G=\mathcal{D}(G)\setminus \mathcal{A}_G$. It is clear  that $D_{G}(1)=|\mathcal{A}_G|+|\mathcal{B}_G|$. By adding $v$, $u$ or both, one can extend every dominating set of $T\setminus u$ to a dominating set of $G_1$ in three different ways. So, we have $|\mathcal{A}_{G_1}|=3D_{T\setminus u}(1)$. If all neighbors of $u$ in $T$ are support vertices in $T$, then $\mathcal{B}_{G_1}=\emptyset$, and the results in $(i)$ and $(ii)$ are clear as $D_{G_1}(1)=3D_{T\setminus u}(1)$ in this case.  So, we may assume that $u$ is adjacent to  exactly one non-support  vertex in $T$, say $u'$. Let $T'$ be the component of $T\setminus u$ which contains $u'$ and let $T^{*}=(T\setminus u)\setminus T'$. Note that if $S$ belongs to $\mathcal{B}_{G_1}$, then $S\cap V(T^{*})$ is a dominating set of $T^{*}$ since every neighbor of $u$ in $T^{*}$ is a support vertex of $T$.  Also, $u$ must be in $S$ and $S$ contains no vertex of $N_{T'}[u']$. Now, the function $f: \mathcal{B}_{G_1}\rightarrow \mathcal{D}(T)$ defined by 
$$
f(S)=
\begin{cases}
S\ \text{if} \ v\notin S\\
(S\setminus \{u,v\})\cup \{u'\}\ \text{if}\ v\in S
\end{cases}
$$
is one-to-one. Hence, $|\mathcal{B}_{G_1}|\leq D_T(1)$ and the result in (i) follows. Let $\lambda (T'\setminus u')$ be the number of dominating sets of $T'\setminus u'$ which do not contain any vertex from $N_{T'}(u')$. Observe that 
\begin{center}
$|\mathcal{B}_{G_1}|=2D_{T^{*}}(1)\lambda(T'\setminus u')\leq 2D_{T^{*}}(1)D_{T'\setminus u'}(1)\leq 2D_{T^{*}}(1)D_{T'}(1)=2D_{T\setminus u}(1)$
\end{center}
and therefore (ii) follows. Lastly,  $|\mathcal{A}_T|\leq 2 D_{T\setminus u}(1)$ because adding the vertex $u$ into a dominating set of $T\setminus u$ may or may not yield a dominating set for $T$.  Moreover, if $S\in \mathcal{B}_T$, then $u$ must be in $S$ and $S$ contains no vertex of $N_{T'}[u']$. Now the function $g: \mathcal{B}_T\rightarrow \mathcal{D}(T\setminus u)$ defined by $g(S)=(S\setminus u)\cup \{u'\}$ is one-to-one and therefore we obtain that $|\mathcal{B}_T|\leq D_{T\setminus u}(1)$. Thus, $D_T(1)=|\mathcal{A}_T|+|\mathcal{B}_T|\leq 3D_{T\setminus u}(1)$ and (iii) is established.
\end{proof}

We are now ready to prove our main result.

\begin{theorem}\label{main} Let $G$ be a forest on $n$ vertices with no isolated vertices. Then, $\operatorname{avd}(G)\leq \frac{2n}{3}$ and, moreover, equality holds if and only if every non-leaf vertex of $G$ is a support vertex with one or two leaf neighbors.
\end{theorem}

\begin{proof} We proceed by strong induction on the number of vertices. If $n=2$, then $G\cong K_2$ and it is clear that $\operatorname{avd}(K_2)=4/3$. We may assume that $n\geq 3$. First suppose that $G$ is connected, that is $G$ is a tree. Let $G$ be a rooted tree and let $v$ be  a support vertex of  maximum distance from the root of $G$. Also let $L_G(v)=\{v_1, \dots , v_k\}$ for some $k\geq 1$. If $V(G)=L_G[v]$, then $G\cong K_{1,k}$ and it easy to see that $\operatorname{avd}(K_{1,k})\leq 2|V(G)|/3$ with equality iff $1\leq k\leq 2$. So we may assume that $V(G)\neq L_G[v]$. Let $T=G\setminus L_G[v]$ and $u\in V(T)$ be the parent of $v$ in $G$. Note that $N_G(v)=L_G(v)\cup\{u\}$, as the chosen support vertex $v$ is of maximum distance from the root of $G$. If $u$ is a support vertex in $G$, then the result follows from Lemma~\ref{support_lemma} and the induction hypothesis. So we may assume that $u$ has no leaf neighbors in $G$. Since $u$ is neither a leaf nor a support vertex of $G$, we have $|V(T)|\geq 3$. Now we shall show that the strict inequality $3D'_G(1)<2nD_G(1)$ holds. 
\begin{eqnarray}
D_G(x) &=& x\left[D_{G/v}(x)+D_{G\setminus \{v, v_k\}}(x)+D_{G\setminus N_G[v]}(x)\right]\\
&=& x\left[D_{T_{(u,k)}}(x)+x^{k-1}D_{T}(x)+D_{T\setminus u}(x)\right]\\
&=& x\left[(x+1)^{k-1}[D_{T_{(u,1)}}(x)+D_{T\setminus u}(x)]-D_{T\setminus u}(x)+x^{k-1}D_T(x)+D_{T\setminus u}(x) \right]\\
&=&x(x+1)^{k-1}\left[D_{G\setminus L_G(v)}(x)+D_{T\setminus u}(x)\right]+x^kD_{T}(x)
\end{eqnarray}

where ($1$) follows from Lemma~\ref{recur}; ($2$) holds as $G/v\cong T_{(u,k)}$, $G\setminus \{v,v_k\}\cong \overbar{K}_{k-1}\cupdot T$ and $G\setminus N_G[v]\cong T\setminus u$; ($3$) follows from Lemma~\ref{glue_lem}; and ($4$) follows since $T_{(u,1)}\cong G\setminus L_G(v)$. Let us write $G_1=G\setminus L_G(v)$, then
\begin{eqnarray*}
D'_G(x)&=& \left[(x+1)^{k-1}+(k-1)x(x+1)^{k-2}\right]\left[D_{G_1}(x)+D_{T\setminus u}(x)\right]\\
&&+x(x+1)^{k-1}\left[D'_{G_1}(x)+D'_{T\setminus u}(x)\right]+kx^{k-1}D_{T}(x)+x^kD'_{T}(x)
\end{eqnarray*}
It follows that
$D_G(1)=2^{k-1}\left[D_{G_1}(1)+D_{T\setminus u}(1)\right]+D_{T}(1)$ and 
$$D'_G(1)=(k+1)2^{k-2}\left[D_{G_1}(1)+D_{T\setminus u}(1)\right]+2^{k-1}\left[D'_{G_1}(1)+D'_{T\setminus u}(1)\right]+kD_{T}(1)+D'_{T}(1).$$
It is not difficult to calculate that $2nD_G(1)-3D'_{G}(1)$ is equal to\\

 $2^{k-1}\big[2(n-k)D_{G_1}(1)-3D'_{G_1}(1)\big]+2^{k-1}\left[2(n-k-2)D_{T\setminus u}(1)-3D'_{T\setminus u}(1)\right]$
 
 $+\left[2(n-k-1)D_T(1)-3D'_T(1)\right]+ 2^{k-2}\left[ (k-3)D_{G_1}(1)+(k+5)D_{T\setminus u}(1)\right]-(k-2)D_{T}(1)$.\\

Each of the subgraphs $T$, $G_1$ and $T\setminus u$ is a proper subforest of $G$ without isolated vertices. So, by the induction hypothesis, we have $3D'_T(1)\leq 2(n-k-1)D_{T}(1)$, $3D'_{G_1}(1)\leq 2(n-k)D_{G_1}(1)$ and $3D'_{T\setminus u}(1)\leq 2(n-k-2)D_{T\setminus u}(1)$. Therefore, in order to show that $3D'_G(1)\leq 2nD_G(1)$ it suffices to prove only that
\begin{equation}\label{main_ineq}
(k-2)D_{T}(1)\leq 2^{k-2}\left[ (k-3)D_{G_1}(1)+(k+5)D_{T\setminus u}(1)\right].
\end{equation}

Observe that $T\setminus u$ has at most  one connected component which is not a star graph because $v$ is of maximum distance from the root of $G$. So, the vertex $u$ has at most one neighbor in $T$ which is not a support vertex of $T$ and Lemma~\ref{3_item_lemma} applies here.
If $k=1$, the inequality \eqref{main_ineq} is $D_{G_1}(1)\leq D_{T}(1)+3D_{T\setminus u}(1)$ and this follows from Lemma~\ref{3_item_lemma}(i). If all neighbors of $u$ in $T$ are support vertices in $T$, then by the proof of Lemma~\ref{3_item_lemma}(i), we have $D_{G_1}(1)=3D_{T\setminus u}(1)$ and therefore strict inequality holds in \eqref{main_ineq} which implies $3D'_G(1)<2nD_G(1)$. If $u$ has a neighbor in $T$ which is not a support vertex of $T$, then  by the induction hypothesis the strict inequality $3D'_T(1)< 2(n-k-1)D_{T}(1)$ holds and again we get $3D'_G(1)<2nD_G(1)$. For $k\geq 2$, we shall show that \eqref{main_ineq} holds strictly.

 If $k=2$,  the strict inequality for \eqref{main_ineq} is $D_{G_1}(1)<7 D_{T\setminus u}(1)$ and this follows from Lemma~\ref{3_item_lemma}(ii). If $k=3$, it is $D_{T}(1)< 16D_{T\setminus u}$ and this is verified by Lemma~\ref{3_item_lemma}(iii). If $k\geq 4$, we have $(k-2)< 2^{k-2}(k-3)$ and $D_T(1)\leq D_{G_1}(1)$ since $T$ is a subgraph of $G_1$. Thus, $3D'_G(1)<2nD_G(1)$  is established for all $k$.

Lastly, suppose that $G$ is a disconnected forest with connected components $H_1,\dots , H_c$ where $c\geq 2$. By the induction hypothesis, for each $i$ we have $\operatorname{avd}(H_i)\leq \frac{2|V(H_i)|}{3}$ with equality iff every non-leaf vertex of $H_i$ is a support vertex with one or two leaf neighbors. Since $\operatorname{avd}(G)=\sum\limits_{i=1}^c\operatorname{avd}(H_i)\leq \sum\limits_{i=1}^c\frac{2|V(H_i)|}{3}=2n/3$, the proof is completed. 

\end{proof}

\section{Concluding remarks}

Every graph $G$ without isolated vertices contains a spanning forest $F$ without isolated vertices, and $F$ can be obtained from $G$ by a succession of non-pendant edge removals. So it would be interesting to investigate how $\operatorname{avd}(G)$ is effected by the removal of a non-pendant edge. In particular we ask the following:
\begin{question}\label{my_ques} In every graph $G$ (which is not a disjoint union of stars or  empty graphs) does there exists a non-pendant edge $e$ of $G$ such that $\operatorname{avd}(G)<\operatorname{avd}(G\setminus e)$?
\end{question}

Observe that an affirmative answer to Question~\ref{my_ques} would yield a proof of Conjecture~\ref{brown_conj} in general because of our Theorem~\ref{main} and the remark above. We also note that Beaton and Brown~\cite{brown_beaton} conjectured that in every non-empty graph $G$, there exists an edge $e$ such that $\operatorname{avd}(G)<\operatorname{avd}(G\setminus e)$ and they verified this conjecture for graphs on up to $7$ vertices. On the other hand, they did not specify a certain property of such existing edge.


\vskip0.3in

\bibliographystyle{elsarticle-num}

\end{document}